\documentclass{amsart}
\usepackage{txfonts}
\usepackage{}
\usepackage{graphicx}
\usepackage{xcolor}
\usepackage{amsmath,amsthm,amstext,amssymb,amsfonts,latexsym}

\newtheorem{thm}{Theorem}[section]

\newtheorem{lem}[thm]{Lemma}
\newtheorem{cor}[thm]{Corollary}
\newtheorem{prop}[thm]{Proposition}

\theoremstyle{definition}

\theoremstyle{remark}
\newtheorem{remark}[thm]{Remark}

\numberwithin{equation}{section}

\allowdisplaybreaks[1]

\begin{document}

\title{Rigidity theorem for compact Bach-flat manifolds with positive constant $\sigma_2$}

\author{Hui-Ya He}
\address{Department of Mathematical Sciences, Tsinghua University, Beijing 100084, P. R. China}
\email{hhy15@mails.tsinghua.edu.cn}




\author{Hai-Ping Fu$^*$}
\address{Department of Mathematics,  Nanchang University, Nanchang 330031, P. R. China}
\email{mathfu@126.com}
\thanks{Supported by National Natural Science Foundations of China \#11761049,  Jiangxi Province
Natural Science Foundation of China \#20171BAB201001.}





\begin{abstract}
  We prove that an $n(\geq 4)$-dimensional compact Bach-flat manifold with positive constant $\sigma_2$ is an Einstein manifold, provided that its Weyl curvature  satisfies a suitable pinching condition.
\end{abstract}

\maketitle

\section{Introduction}

\par Let $(M^n,g) (n\geq3)$ be an $n$-dimensional Riemannian manifold with the Riemannian curvature tensor $Rm=\{R_{ijkl}\}$, the Weyl curvature tensor $W=\{W_{ijkl}\}$, the Ricci curvature tensor $Ric=\{R_{ij}\}$ and the scalar curvature $R$. For any manifold of dimension $n\geq4$, the Bach tensor, introduced by Bach \cite{B}, is defined as
\begin{equation}
B_{ij}\equiv\frac{1}{n-3}\nabla^k\nabla^lW_{ikjl}+\frac{1}{n-2}R^{kl}W_{ikjl},
\end{equation}
where $\nabla$ is the operator of covariant differentiation on $M^n.$    Here and hereafter the Einstein convention of summing over the repeated indices will be adopted. Recall that a metric $g$ is called Bach-flat  and the manifold is called Bach-flat manifold if the Bach tensor vanishes.  It is easy to see that $(M^n,g)(n\geq 4)$ is a Bach-flat manifold,  if it is either a locally conformally flat manifold, or an Einstein manifold.


The curvature pinching phenomenon plays an important role in global differential geometry. Some isolation theorems of the Weyl curvature tensor of positive Ricci Einstein manifolds are given in \cite{{HV},{IS},{S}}, when its $L^{\frac{n}{2}}$-norm is small. Recently, two rigidity theorems of the Weyl curvature tensor of positive Ricci Einstein manifolds are given
in \cite{{C},{FX2},{FX3}}, which improve results due to \cite{{HV},{IS},{S}}.  The second author and Xiao have studied compact manifolds with harmonic curvature to obtain some rigidity results in \cite{{F},{F2},{FX}}. Here when a  Riemannian manifold satisfies $\delta Rm= \{\nabla^lR_{ijkl}\}=0$, we call it a manifold with harmonic curvature.
Bach-flat manifolds have been studied by many authors. For any complete Bach-flat manifold, Kim \cite{K} has studied some rigidity phenomena  and derived that a complete  Bach-flat manifold $M^4$ with nonnegative constant scalar curvature and  positive Yamabe constant is an Einstein manifold if the $L^2$-norm of the trace-free Riemannian curvature tensor $\mathring{Rm}$ is small enough. Later, Chu \cite{Ch} improved Kim's result and showed that $M^4$ is in fact a space of constant curvature under the same assumptions.  For a compact Bach-flat manifold $M^4$ with the positive Yamabe constant, Chang et al. \cite{CQY} proved that $M^4$ is conformal equivalent to the standard four-sphere provided that the $L^2$-norm of the Weyl curvature tensor $W$ is small enough, and also showed that there is only finite diffeomorphism class with a bounded $L^2$-norm of $W$. Peng and the second author \cite{FP} showed that the  compact Bach-flat manifold with positive constant scalar curvature is spherical space form or Einstein manifold under some $L^p$ pinching conditions or some pointwise pinching conditions. For  compact manifolds  with the positive Yamabe constant, Chang et al. \cite{CGY} proved a sharp form of differentiable sphere theorem.

For a Riemannian manifold $(M^n, g)(n\geq3)$, we denote by $\sigma_2(A_g)$ the $2$nd-elementary symmetric
function of the eigenvalues of the so-called Schouten tensor $A_g:=Ric-\frac{R}{2(n-1)}g$ with respect to $g$. Hence
\begin{equation}
\sigma_2(A_g)=\frac{1}{2}\left((trA_g)^2-|A_g|^2\right)=\frac{1}{2}\left(\frac{(n-2)^2}{4n(n-1)}R^2-|\mathring{Ric}|^2\right),\label{00011}
\end{equation}
where $\mathring{Ric}:=Ric-\frac{R}{n}g$ denotes the trace-free Ricci curvature tensor.

Our main result in this paper is the following:
\begin{thm}\label{000001}
Let $(M^n,g)(n\geq4)$ be an n-dimensional  compact Bach-flat  manifold with positive constant $\sigma_2(A_g)$. If
\begin{eqnarray}
|W|^2+\frac{n^2(n-4)}{2(n-2)^3}|\mathring{Ric}|^2< \frac{4n}{(n-2)^3}\sigma_2(A_g),\label{11111}
  \end{eqnarray}
where $\mathring{Ric}:=Ric-\frac{R}{n}g$ is the trace-free Ricci curvature tensor, then $M^n$ is an Einstein manifold.
\end{thm}

\begin{remark}
The pinching condition of Theorem \ref{000001} is optimal. When $M^n=N_1^1\times N_2^{n-1}(c),$ it is easy to compute that $\sigma_2(A_g)=\frac{(n-4)(n-1)(n-2)^2}{8}c^2$ and $|\mathring{Ric}|^2=\frac{(n-1)(n-2)^2}{n}c^2$. In this case the equality in \eqref{11111} holds.
\end{remark}

\begin{cor}\label{00003}
Let $(M^4,g)$ be a 4-dimensional compact Bach-flat manifold with positive scalar curvature and positive constant $\sigma_2(A_g)$. If
\begin{equation}
|W|^2<2\sigma_2(A_g),\label{77777}
\end{equation}
then $M^4$ is isometric to a quotient of the round $\mathbb{S}^4$.
\end{cor}

\begin{cor}\label{000002}
Let $(M^4,g)$ be a 4-dimensional compact locally conformally flat manifold with positive scalar curvature and positive constant $\sigma_2(A_g)$. Then $M^4$  is isometric to a quotient of the round $\mathbb{S}^4$.
\end{cor}

\begin{remark}
In \cite{HLS},  Hu-Li-Simon proved for a compact locally conformally flat manifold $(M^n,g)$ with constant non-zero $\sigma_k(A_g)$ for some $k\in\{2,3,\ldots, n\},$ if the tensor $A_g$ is semi-positive definite, then $(M^n,g)$ is a space form of positive sectional curvature. We enhance this result when $n=4.$
\end{remark}

\begin{cor}\label{00000004}
Let $(M^n,g)(n\geq5)$ be an n-dimensional compact locally conformally flat manifold with positive scalar curvature and positive constant $\sigma_2(A_g)$. If
\begin{eqnarray}
|\mathring{Ric}|^2<\frac{1}{n(n-1)}R^2,\label{1111121}
  \end{eqnarray}
then $M^n$  is isometric to a quotient of the round $\mathbb{S}^n$.
\end{cor}
\begin{remark}
In \cite{FX}, Xiao and the second author proved that an n-dimensional compact locally conformally flat manifold $(M^n,g)(n\geq4)$ with  positive constant scalar curvature is isometric to a quotient of the round $\mathbb{S}^n$, if $\left(\int_{M^n}|\mathring{Ric}|^{\frac n2}\right)^{\frac 2n}<\frac{1}{n(n-1)}Y(M^n, [g])$, where $Y(M^n, [g])$ denotes the Yamabe constant of $(M^n,g)$.
\end{remark}

{Acknowledgement:} The authors are very grateful to Professor Haizhong Li for his guidance and constant support.

\section{Proof of Theorem \ref{000001}}
In what follows, we adopt, without further comment, the moving frame notation with respect to a chosen local orthonormal frame.

Let $(M^n,g)(n\geq3)$ be an $n$-dimensional  compact Riemannian manifold. Decomposing  the Riemannian curvature tensor  into irreducible components (see \cite{Be}, Chapter 1, Section G) yields
\begin{eqnarray}
R_{ijkl}&=&W_{ijkl}+\frac{1}{n-2}(R_{ik}\delta_{jl}-R_{il}\delta_{jk}+R_{jl}\delta_{ik}-R_{jk}\delta_{il})\label{00}\\\nonumber
&&-\frac{R}{(n-1)(n-2)}(\delta_{ik}\delta_{jl}-\delta_{il}\delta_{jk})\\\nonumber
&=&W_{ijkl}+\frac{1}{n-2}(\mathring{R}_{ik}\delta_{jl}-\mathring{R}_{il}\delta_{jk}+\mathring{R}_{jl}\delta_{ik}-\mathring{R}_{jk}\delta_{il})\\\nonumber
&&+\frac{R}{n(n-1)}(\delta_{ik}\delta_{jl}-\delta_{il}\delta_{jk})\\\nonumber
&=&W_{ijkl}+\frac{1}{n-2}(A_{ik}\delta_{jl}-A_{il}\delta_{jk}+A_{jl}\delta_{ik}-A_{jk}\delta_{il}),
\end{eqnarray}
where $R$  is the scalar curvature, $R_{ijkl}$, $W_{ijkl}$, $R_{ij}$, $\mathring{R}_{ij}$  and $A_{ij}$  denote the components of $Rm$, the Weyl curvature tensor $W$,  the Ricci curvature tensor $Ric$, the trace-free Ricci curvature tensor $\mathring{Ric}=Ric-\frac{R}{n}g$  and  the Schouten tensor $A=Ric-\frac{R}{2(n-1)}g$, respectively.

Let $e_1,\ldots,e_n$ be a local orthonormal frame field on $M^n,$ $\omega_1,\ldots,\omega_n$ its dual coframe field, $\phi=\sum_{i,j}\phi_{ij}\omega_i  \otimes      \omega_j$ be a symmetric (0,2)-type tensor defined on $M^n.$   By letting  $\phi_{ij,k}:=\nabla_k\phi_{ij},$ $\phi_{ij,kl}:=\nabla_l\nabla_k\phi_{ij},$ where $\nabla$ is the operator of covariant differentiation on $M^n,$
 we have the following Ricci identities
\begin{eqnarray}
\phi_{ij,kl}-\phi_{ij,lk}=\phi_{mj}R_{mikl}+\phi_{im}R_{mjkl}.\label{11028}
  \end{eqnarray}

The norm of a $(0,4)$-type tensor $T$ is defined as $$|T|^2=|T_{ijkl}|^2=T_{ijkl}T_{ijkl}.$$

By the second Bianchi identity
\begin{equation*}
R_{imkl,j}+R_{imlj,k}+R_{imjk,l}=0,
  \end{equation*}
we have
\begin{equation}
R_{ij,k}-R_{ik,j}=R_{likj,l},\label{11002}
  \end{equation}
and
\begin{equation}
R_{ik,i}=\frac{1}{2}R_{,k}.\label{11003}
  \end{equation}
Then
\begin{equation}
\mathring{R}_{ij,k}-\mathring{R}_{ik,j}=R_{likj,l}+\frac{R_{,j}}{n}\delta_{ik}-\frac{R_{,k}}{n}\delta_{ij},\label{11004}
  \end{equation}
and
\begin{equation}
\mathring{R}_{ik,i}=\frac{n-2}{2n}R_{,k}.\label{11005}
  \end{equation}
By \eqref{11002} and \eqref{11003}, we have
\begin{eqnarray*}
W_{likj,l} &=&  R_{likj,l}     -\frac{1}{n-2}(R_{lk,l}\delta_{ij}-R_{lj,l}\delta_{ik}+R_{ij,l}\delta_{lk}-R_{ik,l}\delta_{lj})+\frac{R_{,l}}{(n-1)(n-2)}(\delta_{lk}\delta_{ij}-\delta_{lj}\delta_{ik})\\
 &=&      R_{ij,k}-R_{ik,j}                -\frac{1}{n-2}(R_{lk,l}\delta_{ij}-R_{lj,l}\delta_{ik}+R_{ij,k}-R_{ik,j})\\
 &&   +\frac{R_{,k}}{(n-1)(n-2)}\delta_{ij}     -\frac{R_{,j}}{(n-1)(n-2)}\delta_{ik}\\
&=&      R_{ij,k}-R_{ik,j}                -\frac{1}{2(n-2)}(R_{,k}\delta_{ij}-R_{,j}\delta_{ik})-\frac{1}{n-2}(R_{ij,k}-R_{ik,j})\\
 &&   +\frac{R_{,k}}{(n-1)(n-2)}\delta_{ij}     -\frac{R_{,j}}{(n-1)(n-2)}\delta_{ik}\\
&=&   \frac{n-3}{n-2}  ( R_{ij,k}-R_{ik,j})              +\frac{n-3}{2(n-1)(n-2)}(R_{,j}\delta_{ik}-R_{,k}\delta_{ij})\\
&=&   \frac{n-3}{n-2}  R_{likj,l}              +\frac{n-3}{2(n-1)(n-2)}(R_{,j}\delta_{ik}-R_{,k}\delta_{ij}).
  \end{eqnarray*}
If $M^n$ is Bach-flat, i.e.,
\begin{equation*}
B_{ij}=\frac{1}{n-3}W_{likj,lk}+\frac{1}{n-2}R_{lk}W_{likj}=0,
  \end{equation*}
then from the above we have
\begin{equation}
  R_{likj,lk}=\frac{1}{2(n-1)}(R_{,kk}\delta_{ij}-R_{,ij})-R_{lk}W_{likj}.\label{11006}
  \end{equation}

In order to prove Theorem  \ref{000001}, we need the following lemmas:
\begin{lem}\label{22001}
Let $(M^n,g)(n\geq4)$ be an n-dimensional  compact Bach-flat Riemannian manifold, then
\begin{equation}
\frac{1}{2}\Delta |\mathring{Ric}|^2 =   |\nabla \mathring{Ric}|^2+     \frac{n-2}{2(n-1)} \mathring{R}_{ij} R_{,ij}- 2 \mathring{R}_{ij}\mathring{R}_{lk}W_{likj}+\frac{n}{n-2}   \mathring{R}_{hj}\mathring{R}_{ij}\mathring{R}_{ih} +\frac{R}{n-1}|\mathring{Ric}|^2.\label{666777}
\end{equation}
\end{lem}
 \begin{proof}
 By using \eqref{00}, \eqref{11004}, \eqref{11005}, \eqref{11006} and the Ricci identity, we get
\begin{eqnarray*}
\frac{1}{2}\Delta |\mathring{Ric}|^2&=&|\nabla \mathring{Ric}|^2+\mathring{R}_{ij} \mathring{R}_{ij,kk}                    \\\label{}
&=&   |\nabla \mathring{Ric}|^2+\mathring{R}_{ij} (       \mathring{R}_{ik,j}+R_{likj,l}+\frac{R_{,j}}{n}\delta_{ik}-\frac{R_{,k}}{n}\delta_{ij}     )_{,k} \nonumber\\
&=&   |\nabla \mathring{Ric}|^2+\mathring{R}_{ij}        \mathring{R}_{ik,jk}             +\mathring{R}_{jk}\frac{R_{,jk}}{n}                +\mathring{R}_{ij}R_{likj,lk}   \nonumber\\
&=&   |\nabla \mathring{Ric}|^2+\mathring{R}_{ij}     (   \mathring{R}_{ik,kj}+ \mathring{R}_{hk}R_{hijk} + \mathring{R}_{ih}R_{hkjk}          )             +\mathring{R}_{jk}\frac{R_{,jk}}{n}                +\mathring{R}_{ij}R_{likj,lk} \nonumber\\
&=&   |\nabla \mathring{Ric}|^2+\frac{n-2}{2n} \mathring{R}_{ij} R_{,ij}+ \mathring{R}_{ij}\mathring{R}_{hk}[W_{hijk}+\frac{1}{n-2}(\mathring{R}_{hj}\delta_{ik}-\mathring{R}_{hk}\delta_{ij}+\mathring{R}_{ik}\delta_{hj}-\mathring{R}_{ij}\delta_{hk}) \nonumber\\
&&+\frac{R}{n(n-1)}(\delta_{hj}\delta_{ik}-\delta_{hk}\delta_{ij}) ]+\mathring{R}_{ij} \mathring{R}_{ih}(\mathring{R}_{hj} +\frac{R}{n}\delta_{hj}   )                   +\mathring{R}_{jk}\frac{R_{,jk}}{n}         \nonumber\\
&&        +\mathring{R}_{ij}[         \frac{1}{2(n-1)}(R_{,kk}\delta_{ij}-R_{,ij})-\mathring{R}_{lk}W_{likj}                     ] \nonumber\\
&=&   |\nabla \mathring{Ric}|^2+    \frac{n-2}{2(n-1)} \mathring{R}_{ij} R_{,ij}- 2 \mathring{R}_{ij}\mathring{R}_{lk}W_{likj}+\frac{n}{n-2}   \mathring{R}_{hj}\mathring{R}_{ij}\mathring{R}_{ih} +\frac{R}{n-1}|\mathring{Ric}|^2.
  \end{eqnarray*}
This completes the proof of Lemma 2.1.
\end{proof}

\begin{lem}\label{22002}
Let $(M^n,g)(n\geq4)$ be an n-dimensional  compact Bach-flat Riemannian manifold with positive constant $\sigma_2(A_g)$, then
\begin{eqnarray*}
0\geq \frac{n}{n-2}  \int_{M^n} \mathring{R}_{hj}\mathring{R}_{ij}\mathring{R}_{ih} +\frac{1}{n-1}\int_{M^n}R|\mathring{Ric}|^2- 2\int_{M^n}\mathring{R}_{ij}\mathring{R}_{lk}W_{likj}.
\end{eqnarray*}

\end{lem}
\begin{proof}
We compute
  \begin{eqnarray}
|\nabla A|^2&=&|\nabla Ric|^2-\frac{1}{n-1}|\nabla R|^2+\frac{|\nabla R|^2}{4(n-1)^2}n\label{333444}\\\nonumber
&=&|\nabla Ric|^2-\frac{3n-4}{4(n-1)^2}|\nabla R|^2\\\nonumber
&=&|\nabla \mathring{Ric}|^2+\frac{(n-2)^2}{4n(n-1)^2}|\nabla R|^2,
  \end{eqnarray}
and
\begin{equation}
|\nabla trA|^2=\frac{(n-2)^2}{4(n-1)^2}|\nabla R|^2.\label{444555}
\end{equation}
Since $\sigma_2(A_g)$ is a
positive constant,  the  inequality of Kato type due to Hu-Li-Simon \cite{HLS}, Li \cite{L} and Simon \cite{Simon}, i.e.,
\begin{equation}
|\nabla A|^2\geq|\nabla trA|^2\label{555666}
\end{equation}
holds. From \eqref{333444} and \eqref{444555}, \eqref{555666} implies that
\begin{equation}|\nabla \mathring{Ric}|^2\geq\frac{(n-2)^2}{4n(n-1)}|\nabla R|^2.\label{777888}
\end{equation}

Integrating \eqref{666777} by parts on $M^n$ and using \eqref{11005} and \eqref{777888} we have
\begin{eqnarray*}
0&=&\int_{M^n}|\nabla \mathring{Ric}|^2+     \frac{n-2}{2(n-1)} \int_{M^n}\mathring{R}_{ij} R_{,ij}- 2\int_{M^n} \mathring{R}_{ij}\mathring{R}_{lk}W_{likj}+\frac{n}{n-2}  \int_{M^n} \mathring{R}_{hj}\mathring{R}_{ij}\mathring{R}_{ih} +\int_{M^n}\frac{R}{n-1}|\mathring{Ric}|^2\\
&=&\int_{M^n}|\nabla \mathring{Ric}|^2-\frac{n-2}{2(n-1)} \int_{M^n}\mathring{R}_{ij,j} R_{,i}- 2\int_{M^n} \mathring{R}_{ij}\mathring{R}_{lk}W_{likj}+\frac{n}{n-2}  \int_{M^n} \mathring{R}_{hj}\mathring{R}_{ij}\mathring{R}_{ih} +\int_{M^n}\frac{R}{n-1}|\mathring{Ric}|^2\\
&=&\int_{M^n}|\nabla \mathring{Ric}|^2-     \frac{(n-2)^2}{4n(n-1)} \int_{M^n}|\nabla R|^2- 2\int_{M^n} \mathring{R}_{ij}\mathring{R}_{lk}W_{likj}+\frac{n}{n-2}  \int_{M^n} \mathring{R}_{hj}\mathring{R}_{ij}\mathring{R}_{ih} +\int_{M^n}\frac{R}{n-1}|\mathring{Ric}|^2\\
&\geq&- 2\int_{M^n} \mathring{R}_{ij}\mathring{R}_{lk}W_{likj}+\frac{n}{n-2}  \int_{M^n} \mathring{R}_{hj}\mathring{R}_{ij}\mathring{R}_{ih} +\int_{M^n}\frac{R}{n-1}|\mathring{Ric}|^2.\\
 \end{eqnarray*}
 This completes the proof of Lemma 2.2.
 \end{proof}

\begin{lem}\label{22003}
Let $(M^n,g)(n\geq4)$ be an n-dimensional   Riemannian manifold, then
$$\left|-W_{ijkl}\mathring{R}_{ik}\mathring{R}_{jl}+\frac{n}{2(n-2)}\mathring{R}_{ij}\mathring{R}_{jk}\mathring{R}_{ik}\right|
\leq\sqrt{\frac{n-2}{2(n-1)}}|\mathring{Ric}|^2\left(|W|^2+\frac{n}{2(n-2)}|\mathring{Ric}|^2\right)^{\frac{1}{2}}.$$
\end{lem}
\begin{remark}
We follow these proofs of Proposition 2.1 in \cite{C} and Lemma 4.7 in \cite{Bo} to prove this lemma which is proved in \cite{FP}. For completeness, we also write it out.
\end{remark}
\begin{proof}
First of all we have
$$(\mathring{Ric}\circledwedge g)_{ijkl}=\mathring{R}_{ik}g_{jl}-\mathring{R}_{il}g_{jk}+\mathring{R}_{jl}g_{ik}-\mathring{R}_{jk}g_{il},$$
$$(\mathring{Ric}\circledwedge \mathring{Ric})_{ijkl}=2(\mathring{R}_{ik}\mathring{R}_{jl}-\mathring{R}_{il}\mathring{R}_{jk}),$$
where $\circledwedge$ denotes the Kulkarni-Nomizu product.
An easy computation shows
$$W_{ijkl}\mathring{R}_{ik}\mathring{R}_{jl}=\frac{1}{4}W_{ijkl}(\mathring{Ric}\circledwedge \mathring{Ric})_{ijkl},$$
$$\mathring{R}_{ij}\mathring{R}_{jk}\mathring{R}_{ik}=-\frac{1}{8}(\mathring{Ric}\circledwedge g)_{ijkl}(\mathring{Ric}\circledwedge \mathring{Ric})_{ijkl}.$$
Hence we get the following equation
\begin{equation}
-W_{ijkl}\mathring{R}_{ik}\mathring{R}_{jl}+\frac{n}{2(n-2)}\mathring{R}_{ij}\mathring{R}_{jk}\mathring{R}_{ik}=-\frac{1}{4}\left(W+\frac{n}{4(n-2)}\mathring{Ric}\circledwedge g\right)_{ijkl}(\mathring{Ric}\circledwedge \mathring{Ric})_{ijkl}.\label{87654}
\end{equation}
Since $\mathring{Ric}\circledwedge \mathring{Ric}$ has the same symmetries with the Riemannian curvature tensor, it can be orthogonally decomposed as
$$\mathring{Ric}\circledwedge \mathring{Ric}=T+V'+U'.$$
Here $T$ is totally trace-free, and
$$V'_{ijkl}=-\frac{2}{n-2}\left(\mathring{Ric}^2\circledwedge g\right)_{ijkl}+\frac{2}{n(n-2)}|\mathring{Ric}|^2(g\circledwedge g)_{ijkl},$$
$$U'_{ijkl}=-\frac{1}{n(n-1)}|\mathring{Ric}|^2(g\circledwedge g)_{ijkl},$$
where $\left(\mathring{Ric}^2\right)_{ik}=\mathring{R}_{ip}\mathring{R}_{kp}$. Taking the squared norm we obtain
$$|\mathring{Ric}\circledwedge \mathring{Ric}|^2=8|\mathring{Ric}|^4-8|\mathring{Ric}^2|^2,$$
$$|V'|^2=\frac{16}{n-2}|\mathring{Ric}^2|^2-\frac{16}{n(n-2)}|\mathring{Ric}|^4,$$
$$|U'|^2=\frac{8}{n(n-1)}|\mathring{Ric}|^4.$$
In particular, one has
\begin{equation}
|T|^2+\frac{n}{2}|V'|^2=|\mathring{Ric}\circledwedge \mathring{Ric}|^2+\frac{n-2}{2}|V'|^2-|U'|^2=\frac{8(n-2)}{n-1}|\mathring{Ric}|^4.\label{98765}
\end{equation}
We now estimate the right hand side of \eqref{87654}. Using  \eqref{98765}, Cauchy-Schwarz inequality and  the fact that $W$ and $T$ are totally trace-free we obtain
\begin{eqnarray*}
\left|\left(W+\frac{n}{4(n-2)}\mathring{Ric}\circledwedge g\right)_{ijkl}(\mathring{Ric}\circledwedge \mathring{Ric})_{ijkl}\right|^2
&=&\left|\left(W+\frac{n}{4(n-2)}\mathring{Ric}\circledwedge g\right)_{ijkl}(T+V')_{ijkl}\right|^2\\
&=&\left|\left(W+\frac{\sqrt{2n}}{4(n-2)}\mathring{Ric}\circledwedge g\right)_{ijkl}\left(T+\sqrt{\frac{n}{2}}V'\right)_{ijkl}\right|^2\\
&\leq&\left|W+\frac{\sqrt{2n}}{4(n-2)}\mathring{Ric}\circledwedge g\right|^2\left(|T|^2+\frac{n}{2}|V'|^2\right)\\
&=&\frac{8(n-2)}{n-1}|\mathring{Ric}|^4\left(|W|^2+\frac{n}{2(n-2)}|\mathring{Ric}|^2\right).
\end{eqnarray*}
This estimate together with \eqref{87654} concludes this proof.
\end{proof}

\begin{proof}[\textbf{Proof of Theorem \ref{000001}}]
By (1.2), the pinching condition (1.3) in Theorem 1.1 is equivalent to
\begin{eqnarray}
|W|^2+\frac{n}{2(n-2)}|\mathring{Ric}|^2<\frac{R^2}{2(n-1)(n-2)}.\label{300001}
  \end{eqnarray}
By Lemma \ref{22002} and Lemma \ref{22003}, we obtain
\begin{equation}
0\geq \int_{M^n}\left[\frac{R}{n-1}-2\sqrt{\frac{n-2}{2(n-1)}}\left(|W|^2+\frac{n}{2(n-2)}|\mathring{Ric}|^2\right)^{\frac{1}{2}}\right]|\mathring{Ric}|^2.\label{3000002}
\end{equation}
Combining \eqref{300001} with \eqref{3000002}, we get that $\mathring{Ric}=0$, i.e., $M^n$ is an Einstin manifold. We finish the proof of Theorem \ref{000001}.
\end{proof}

\begin{proof}[\textbf{Proof of Corollary 1.3}]
When $n=4$, the pinching condition (1.3) in Theorem 1.1 is reduced to  (1.4) in Corollary 1.3. By Theorem 1.1, $M^4$ is Einstein. Thus by (1.2), (1.4) is equivalent to
\begin{eqnarray}
|W|^2< \frac{R^2}{12}.\label{30001}
\end{eqnarray}
By Theorem 1.8 in \cite{F}, we obtain that
$M^4$ is isometric to a quotient of the round $\mathbb{S}^4$. We finish the proof of Theorem \ref{000001}.
\end{proof}

\begin{proof}[\textbf{Proof of Corollary \ref{000002}}]
According to $\sigma_2(A_g)$ is a positive constant and the fact that a locally conformally flat manifold is also a Bach-flat manifold,   $(M^4,g)$ satisfies condition \eqref{77777} in Corollary  \ref{00003}, then we know that  $M^4$ is isometric to $\mathbb{S}^4$. Hence we complete the proof of Corollary \ref{000002}.
\end{proof}

\begin{proof}[\textbf{Proof of Corollary \ref{00000004}}]
From \eqref{1111121} and  \eqref{00011}, we can easily get
\begin{eqnarray*}
|\mathring{Ric}|^2< \frac{8}{n(n-4)}\sigma_2(A_g).
\end{eqnarray*}
Since $(M^n,g)$ is a compact locally conformally flat manifold, we have $|W|=0$, then $(M^n,g)$ satisfies condition \eqref{11111} in Theorem \ref{000001}. According to the fact that a locally conformally flat manifold is also a Bach-flat manifold, by use of Theorem \ref{000001}, we know that $M^n$ is an Einstein manifold. Then we can get the conclusion that $(M^n,g)$ is isometric to a quotient of the round $\mathbb{S}^n$. Hence we complete the proof of Corollary \ref{00000004}.
\end{proof}

From the proofs of Lemma 2.2 and Theorem 1.1, we have
\begin{prop}\label{76543}
Let $(M^n, g)(n\geq4)$  be an $n$-dimensional   compact Bach-flat manifold  with  positive scalar curvature. If
\begin{equation}
\int_{M^n}|\nabla \mathring{Ric}|^2\geq\frac{(n-2)^2}{4n(n-1)}\int_{M^n}|\nabla R|^2,\label{65432}
\end{equation}
and
\begin{equation}
|W|^2+\frac{n}{2(n-2)}|\mathring{Ric}|^{2}
<\frac{1}{{2(n-2)(n-1)}}R^2,
\end{equation}
then $M^n$ is an Einstein manifold.
\end{prop}
\begin{remark}
For  Bach-flat manifolds  with positive constant scalar curvature, \eqref{65432} naturally holds. Proposition \ref{76543} improves Theorem 3 in \cite{FP}.
\end{remark}

\end{document}